\documentclass[12pt,twoside,a4paper]{amsart}
\usepackage{amssymb}
\usepackage{amsmath}

\def\imult{\lrcorner}
\def\dbar{\bar\partial}
\def\C{{\mathbb C}}
\def\Cn{\C^n}
\def\PM{{\mathcal{PM}}}
\def\Hom{{\rm Hom\,}}
\def\codim{{\rm codim\,}}
\def\Im{{\rm Im\, }}
\def\Ker{{\rm Ker\,  }}
\def\Ok{{\mathcal O}}
\def\Ow{{\tilde{\Ok}}}
\def\Re{{\rm Re\,  }}
\def\reg{{{\rm reg}}}
\def\sing{{{\rm sing}}}
\DeclareMathOperator{\Id}{Id}
\DeclareMathOperator{\supp}{supp}
\DeclareMathOperator{\ann}{ann}
\DeclareMathOperator{\rank}{rank}
\DeclareMathOperator{\depth}{depth}
\DeclareMathOperator{\hd}{dh}
\DeclareMathOperator{\Tor}{Tor}

\newtheorem{thm}{Theorem}
\newtheorem{lma}[thm]{Lemma}
\newtheorem{cor}[thm]{Corollary}
\newtheorem{prop}[thm]{Proposition}

\theoremstyle{definition}

\newtheorem{df}[thm]{Definition}

\theoremstyle{remark}

\newtheorem{preremark}[thm]{Remark}
\newtheorem{preex}[thm]{Example}

\newenvironment{remark}{\begin{preremark}}{\end{preremark}}
\newenvironment{ex}{\begin{preex}}{\end{preex}}

\begin{document}

\title{On the duality theorem on an analytic variety}

\date{\today}

\author{Richard L\"ark\"ang}

\address{Richard L\"ark\"ang\\ Department of
  Mathematics\\Chalmers University of Technology and the University of
  Gothenburg\\412 96 G\"oteborg\\Sweden}

\email{larkang@chalmers.se}


\maketitle

\begin{abstract}
    The duality theorem for Coleff-Herrera products on a complex manifold says that
    if $f = (f_1,\dots,f_p)$ defines a complete intersection, then the 
    annihilator of the Coleff-Herrera product $\mu^f$ equals (locally) the
    ideal generated by $f$. This does not hold unrestrictedly on an analytic variety $Z$.
    We give necessary, and in many cases sufficient conditions for when the
    duality theorem holds. These conditions are related to how the zero set of
    $f$ intersects certain singularity subvarieties of the sheaf $\Ok_Z$.
\end{abstract}

\section{Introduction}

    Let $f = (f_1,\dots,f_p)$ be a tuple of holomorphic functions on an analytic variety $Z$,
    where we throughout the article will assume that $Z$ has pure dimension.
    The \emph{Coleff-Herrera product} of $f$, as introduced in \cite{CH}, can be defined by
    \begin{equation} \label{eqchpdef}
        \mu^f = \dbar \frac{1}{f_p}\wedge \dots \wedge \dbar \frac{1}{f_1} . \varphi :=
        \left.\int_{Z} \frac{\dbar |f_p|^{2\lambda}\wedge \dots \wedge \dbar |f_1|^{2\lambda}}{f_p\dots f_1} \wedge \varphi \right|_{\lambda = 0}.
    \end{equation}
    Here, $\varphi$ is a test form, and the integral on the right-hand side is analytic in $\lambda$ for
    $\Re \lambda \gg 0$, and has an analytic continuation to $\lambda = 0$, and $|_{\lambda = 0}$ denotes this value.
    We denote the Coleff-Herrera product of $f$ either by $\mu^f$, or by $\dbar (1/f_1)\wedge\dots\wedge\dbar(1/f_p)$.
    The definition \eqref{eqchpdef} is different from the original one, but in the case we focus on here, that
    $f$ defines a \emph{complete intersection}, i.e., that $\codim Z_f = p$, various different definitions including this
    definition and the original definition by Coleff and Herrera coincide, also on a singular variety, see \cite{LS}.
   
    If $f$ defines a complete intersection, the \emph{duality theorem}, proven by Dickenstein and Sessa, \cite{DS},
    and Passare, \cite{P}, gives a close relation between the Coleff-Herrera product of $f$ and the ideal $\mathcal{J}(f_1,\dots,f_p)$
    generated by $f$. This is done by means of the annihilator, $\ann \mu^f$, of $\mu^f$, i.e., the holomorphic
    functions $g$ such that $g \mu^f = 0$.
    \begin{thm}
        Let $f = (f_1,\dots,f_p)$ be a holomorphic mapping on a complex manifold defining a complete intersection.
        Then locally, $\mathcal{J}(f_1,\dots,f_p) = \ann \mu^f$.
    \end{thm}

    The Coleff-Herrera product of a holomorphic mapping is a current on $Z$. Currents on singular varieties
    can be defined in a similar way as on manifolds, i.e., as linear functionals on test-forms,
    see for example \cite{L}. However, currents on $Z$ also has a characterization in terms of currents in the ambient space:
    If $i : Z \to \Omega$ is the inclusion, $\codim Z = k$, and $\mu$ is a $(p,q)$-current on $Z$, then
    $i_* \mu$ is a $(k + p,k + q)$-current on $\Omega$ that vanishes on all forms that vanish on $Z$.
    Conversely, if $T$ is a $(k + p, k + q)$-current on $\Omega$, that vanishes on all forms that vanish on $Z$,
    then $T$ defines a unique $(p,q)$-current $T'$ on $Z$ such that $i_* T' = T$.
    When we consider the Coleff-Herrera product in the ambient space, i.e., $i_* \mu^f$, we will denote it by
    \begin{equation*}
        \dbar \frac{1}{f_p}\wedge\dots\wedge \dbar \frac{1}{f_1} \wedge [Z],
    \end{equation*}
    and in fact, by analytic continuation, it can be defined by
    \begin{equation} \label{eqmufambientreg}
        \dbar\frac{1}{f_p}\wedge \dots \wedge \dbar \frac{1}{f_1} \wedge [Z] =
        \left.\frac{\dbar |f_p|^{2\lambda}\wedge \dots \wedge \dbar |f_1|^{2\lambda}}{f_p\dots f_1} \wedge [Z] \right|_{\lambda = 0}.
    \end{equation}

    On an analytic variety, one can find rather simple examples of functions annihilating the Coleff-Herrera
    product of a complete intersection without lying in the ideal. However, we have an inclusion in one of
    the directions, see \cite{CH}, Theorem~1.7.7.
    \begin{thm} \label{thmanninclusion}
        If $f = (f_1,\dots,f_p)$ are holomorphic on $Z$, defining a complete intersection,
        then $\mathcal{J}(f_1,\dots,f_p) \subseteq \ann \mu^f$.
    \end{thm}

    In this article, we discuss this inclusion, and give conditions for when the inclusion is an equality,
    and when the inclusion is strict.

    Throughout this article, we will only discuss the duality theorem for strongly holomorphic functions on $Z$,
    i.e., functions $f$ on $Z$, which are locally the restriction of holomorphic functions in the ambient space,
    denoted $f \in \Ok(Z)$.
    When we say holomorphic functions, we refer to strongly holomorphic functions.
    However, we will sometimes refer to them as strongly holomorphic functions, to make a distinction
    to weakly holomorphic, which we use in the introduction to provide examples.
    Recall that a function $f : Z_\reg \to \C$ is \emph{weakly holomorphic} on $Z$, denoted $f \in \Ow(Z)$, if
    $f$ is holomorphic on $Z_\reg$, and $f$ is locally bounded at $Z_\sing$.
    Recall also that a germ of a variety, $(Z,z)$, is said to be \emph{normal} if $\Ok_{Z,z} = \Ow_{Z,z}$,
    and that the \emph{normalization} of a variety $Z$ is the unique (up to analytic isomorphism) normal variety
    $Z'$ together with a finite proper surjective holomorphic map $\pi : Z' \to Z$ such that
    $\pi|_{Z'\setminus\pi^{-1}(Z_\sing)} : Z' \setminus \pi^{-1}(Z_\sing) \to Z_\reg$ is a biholomorphism,
    see for example \cite{Dem}, Section II.7.

    One of the reasons we do not have equality in Theorem~\ref{thmanninclusion} is because of weakly holomorphic functions,
    namely if $f = (f_1,\dots,f_p)$ is strongly holomorphic and defining a complete intersection,
    and $g = \sum a_i f_i$ is strongly holomorphic while the functions
    $a_i$ are only weakly holomorphic, then by Theorem~4.3 in \cite{L} (the analogue of Theorem~\ref{thmanninclusion} for weakly
    holomorphic functions), $g \mu^f = 0$,
    but it might very well happen that the $a_i$ cannot be chosen to be strongly holomorphic.
    For example, let $Z = \{ z^3 = w^2 \} \subseteq \C^2$, which has normalization $\pi(t) = (t^2,t^3)$, and let $f \in \Ow(Z)$
    be such that $\pi^* f = t$.
    Then $f^2 = z$ and $f^3 = w$ on $Z$, so that $f^2, f^3 \in \Ok(Z)$ and $f^3 \dbar (1/f^2) = 0$
    (note that since $f^2$ is strongly holomorphic on $Z$, we see this as a current on $Z$, as explained above),
    while $f^3 \neq g f^2$ for any $g \in \Ok(Z)$, since $f \notin \Ok(Z)$.
    That $f^3 \dbar (1/f^2) = 0$ can be seen either by going back to the normalization, where we get $t^3 \dbar (1/t^2)$,
    which is $0$ by the (smooth) duality theorem, or by seeing it as a current in the ambient space, and using the Poincar\'e-Lelong
    formula as in Example~\ref{ex1} below.

    Let us now consider a germ of a normal variety $(Z,z)$, and the Coleff-Herrera product of one holomorphic function.
    Assume that $g \in \ann \dbar(1/f)$. Since $\dbar (1/f)$ is just $\dbar$ of $1/f$ in the current sense
    and $g$ is holomorphic, we get that
    \begin{equation*}
        \dbar\left(g \frac{1}{f}\right) = 0.
    \end{equation*}
    In the smooth case, by regularity of the $\dbar$-operator on $0$-currents, $g (1/f)$ would be a holomorphic
    function. This will not hold in general on a singular space (as the example above shows). However, we
    get that $g/f \in \Ok(Z_\reg)$. If $(Z,z)$ is normal, then $\codim (Z_\sing,z) \geq 2$ in $Z$, and
    any function holomorphic on an analytic variety outside some subvariety of codimension $\geq 2$
    is locally bounded, see \cite{Dem}, Proposition~II.6.1. Thus, $g/f$ is weakly holomorphic,
    and since $(Z,z)$ is normal, $g/f \in \Ok_{Z,z}$, i.e., $g \in \mathcal{J}(f)$.
    Combined with Theorem~\ref{thmanninclusion}, we get that the duality theorem holds for the
    Coleff-Herrera product of one holomorphic function on $(Z,z)$ if it is normal.

    Assume now that $(Z,z)$ is not normal. Then, there exists $\phi \in \Ow_{Z,z} \setminus \Ok_{Z,z}$.
    Since weakly holomorphic functions are meromorphic, we can write $\phi = g/h$
    for some strongly holomorphic functions $g$ and $h$.
    Then $g \dbar (1/h) = 0$, by Theorem~4.3 in \cite{L} (the analogue of Theorem~\ref{thmanninclusion} for weakly
    holomorphic functions). However, since $g/h = \phi \in \Ow_{Z,z} \setminus \Ok_{Z,z}$,
    $g \notin \mathcal{J}(h)$ (in $\Ok_{Z,z}$).

    Hence, in the case of the Coleff-Herrera product of one single holomorphic function on a germ of an analytic variety $(Z,z)$,
    we get that the duality theorem holds for all $f$ if and only if $(Z,z)$ is normal.
    The next example shows that this characterization does not extend to tuples of holomorphic functions.

    \begin{ex} \label{ex1}
        Let $Z = \{ z_1^2 + \cdots + z_k^2 = 0 \} \subseteq \C^k$, where $k \geq 3$.
        Then $Z$ is normal since $Z$ is a reduced complete intersection with $Z_\sing = \{ 0 \}$,
        and a reduced complete intersection is normal if and only if $\codim Z_\sing \geq 2$
        (see the discussion after Definition~\ref{defpduality}).
        Let $\mu = \dbar (1/z_{k-1})\wedge\cdots\wedge\dbar (1/z_1)$ (seen as a current on $Z$).
        We claim that $z_k \mu = 0$. To see this, we consider this as a current in the ambient space,
        i.e., $i_*(z_k \mu)$, and use the Poincar\'e-Lelong formula,
        \begin{equation*}
            i_*(z_k \mu) = z_k\dbar \frac{1}{z_{k-1}}\wedge\cdots\wedge\dbar \frac{1}{z_1}\wedge
            \frac{1}{2\pi i}\dbar \frac{1}{z_1^2 + \cdots + z_k^2}\wedge d(z_1^2 + \cdots + z_k^2).
        \end{equation*}
        Then, $z_kdz_i^2 = 2z_iz_k dz_i$ and $z_iz_k \in \mathcal{J}(z_1,\cdots,z_{k-1},z_1^2+\cdots + z_k^2)$
        for $i = 1,\dots,k$, so each such term annihilates the current by Theorem~\ref{thmanninclusion}.
        However, $z_k \not\in \mathcal{J}(z_1,\cdots,z_{k-1})$ in $\Ok(Z)$.
    \end{ex}

    We will show that depending on certain singularity subvarieties of the analytic sheaf $\Ok_Z$,
    compared to the zero set of $f$, we can give sufficient (and in many cases necessary) conditions
    for when the duality theorem holds on an analytic variety.
    This condition can be seen as a generalization of normality, coinciding with the usual notion
    of normality in the case $p = 1$.

    Given a coherent ideal sheaf $\mathcal{J}$, there exists locally a finite free resolution
    \begin{equation} \label{eqfreeres}
        0 \to \Ok(E_N) \xrightarrow[]{\varphi_N} \Ok(E_{N-1}) \to \cdots \xrightarrow[]{\varphi_1} \Ok(E_0)
    \end{equation}
    where $\Ok(E_k)$ is the sheaf associated to the vector bundle $E_k$.
    We define $Z_k$ as the set of points where $\varphi_k$ does not have optimal rank.
    If $Z = Z(\mathcal{J})$ and $p = \codim Z$, then $Z_1 = \dots = Z_p = Z$
    and $Z_{k+1} \subseteq Z_k$, see \cite{E}, Corollary 20.12.
    If $\mathcal{J} = \mathcal{J}_Z$, the ideal of holomorphic functions vanishing on $Z$,
    then we define
    \begin{equation} \label{eqzkintrinsic}
        Z^0 := Z_\sing \quad \text{and} \quad Z^k := Z_{p + k} \quad \text{ for } k \geq 1,
    \end{equation}
    where $p = \codim Z$. These sets are in fact independent of the choice of resolution by
    the uniqueness of minimal free resolutions in a local Noetherian ring, and from Lemma~3.1
    and the remark following it in \cite{AW3},
    $Z^k$ are independent of the local embedding of $Z$ into $\Cn$.
    Hence they are intrinsic subvarieties of $Z$.
    We will use the convention that $\codim Z^k$ refers to the codimension in $Z$,
    while by $\codim Z_k$, we refer to the codimension in the ambient space.

    \begin{thm} \label{annmuf}
        Let $f = (f_1,\dots,f_p)$ be a holomorphic mapping on a germ of an analytic variety $(Z,z)$ defining a
        complete intersection. If $\codim (Z^k \cap Z_f) \geq k + p + 1$ for $k \geq 0$,
        then $\ann \mu^f = \mathcal{J}(f_1,\dots,f_p)$.
    \end{thm}

    The proof of Theorem~\ref{annmuf} is in Section~\ref{sectproofannmuf}.

    One might conjecture that this equality of the annihilator and the ideal holds if and only if
    the conditions in the theorem are satisfied. We have not been able to prove this in this
    generality, but have focused on a slightly weaker formulation of it.
    To do this, we introduce the notion of $p$-duality for an analytic variety.
    \begin{df} \label{defpduality}
        If $(Z,z)$ is a germ of an analytic variety, we say that $(Z,z)$ has \emph{$p$-duality} if for all
        $f = (f_1,\dots,f_p) \in \Ok_{Z,z}^{\oplus p}$ defining a complete intersection, we have $\ann \mu^f = \mathcal{J}(f_1,\dots,f_p)$.
    \end{df}
    Theorem~\ref{annmuf} implies the following statement:
    \begin{equation} \label{pdualitystatement}
        \tag{$*$}
        (Z,z) \text{ has $p$-duality if } \codim Z^k \geq p + k + 1, \text{ for } k \geq 0.
    \end{equation}
    We believe that the converse of \eqref{pdualitystatement} holds, and we will discuss this throughout
    the rest of this introduction. We show that indeed, in many cases, the converse of \eqref{pdualitystatement}
    holds, and if the condition in \eqref{pdualitystatement} is not a precise condition for $p$-duality,
    it is at least very close to being so.

    We saw above that $1$-duality is equivalent to that $Z$ is normal.
    The condition $\codim Z^k \geq k + 2$ in \eqref{pdualitystatement} is exactly the condition that $Z$
    is normal. This is proved in \cite{M}, but can also be seen using the conditions R1 and S2 in
    Serre's criterion for normality. Indeed, one can verify that the conditions R1 and S2 are equivalent
    to the condition $\codim Z^k \geq k + 2$.
    Thus, the converse of \eqref{pdualitystatement} holds when $p = 1$.

    Recall that a germ $(Z,z)$ is said to be \emph{Cohen-Macaulay} if the ring $\Ok/\mathcal{J}_{Z,z}$ is Cohen-Macaulay.
    More concretely, this means that $\Ok/\mathcal{J}_{Z,z}$ has a free resolution of length $p = \codim (Z,z)$.
    Equivalently, $Z^k = \emptyset$ for $k \geq 1$. Hence, if $(Z,z)$ is Cohen-Macaulay,
    the condition $\codim Z^k \geq p + k$ for $k \geq 0$ becomes just $\codim Z_\sing \geq p$.
    In case $(Z,z)$ is Cohen-Macaulay, the converse of \eqref{pdualitystatement} holds.

    \begin{prop} \label{propcounterex1}
        Assume that $(Z,z)$ is Cohen-Macaulay and that $\codim Z_\sing = k$.
        If $q \geq k$, then there exists $f = (f_1,\dots,f_q) \in \Ok_{Z,w}^{\oplus q}$,
        for some $w$ arbitrarily close to $z$, defining a complete intersection, and
        $g \in \Ok_{Z,w}$ such that $g \in \ann \mu^f$, but $g \notin \mathcal{J}(f_1,\dots,f_q)$.
    \end{prop}

    \begin{remark}
        In general, we need to move to a nearby germ in order to find the counterexample,
        however, if $Z_\sing$ is a complete intersection in $Z$, we can take $w = z$.
    \end{remark}

    In particular, if $(Z,z)$ is a reduced complete intersection, then $(Z,z)$ is Cohen-Macaulay
    since the Koszul complex is a free resolution of length $\codim (Z,z)$, see \cite[p. 688]{GH}.

    In Example~\ref{ex1}, $(Z,0)$ is Cohen-Macaulay (since it is a reduced complete intersection)
    and $Z_\sing = \{ 0 \}$, which has codimension $k-1$ in $(Z,0)$.
    Proposition~\ref{propcounterex1} then says that there exists a complete intersection $f = (f_1,\dots,f_{k-1})$
    and $g \notin \mathcal{J}(f_1,\dots,f_{k-1})$ such that $g \in \ann \mu^f$.
    Then $f = (z_1,\dots,z_{k-1})$ and $g = z_k$ is exactly such an example,
    while for any complete intersection of codimension $ < k-1$, the duality theorem holds
    by Theorem~\ref{annmuf}.

    If $(Z,z)$ is not Cohen-Macaulay, we get the converse of \eqref{pdualitystatement}
    only for the least $p$ such that the condition in \eqref{pdualitystatement} is not satisfied.

    \begin{prop} \label{propcounterex2}
        Assume that $(Z,z)$ satisfies $\codim Z^k \geq k + p$ for all $k \geq 0$, with equality for some $k\geq 1$.
        Then there exists $f = (f_1,\dots,f_p) \in \Ok_{Z,z}^{\oplus p}$ defining a complete intersection,
        and $g \in \Ok_{Z,z}$, such that $g \in \ann \mu^f$, but $g \notin \mathcal{J}(f_1,\dots,f_p)$.
    \end{prop}

    If $p=1$, then the weakly holomorphic functions give rise to counterexamples as described above.

    The proofs of Proposition~\ref{propcounterex1} and Proposition~\ref{propcounterex2}
    are in Section~\ref{sectproofce1} and Section~\ref{sectproofce2} respectively.
    To prove Proposition~\ref{propcounterex1}, we use Theorem~\ref{thmxivanish}, which says
    that there exists a tuple $\xi$ of holomorphic $(p,0)$-forms such that
    \begin{equation} \label{eqintcurrepr}
        [Z] = \sum \xi_i \wedge R^Z_i,
    \end{equation}
    where $[Z]$ is the integration current on $Z$, and $R^Z = (R^Z_1,\dots,R^Z_N)$ is a tuple of
    currents such that $\mathcal{J}_Z = \cap_{i=1}^N \ann R^Z_i$, and the current $R^Z$ is
    defined by means of a free resolution of $\Ok/\mathcal{J}_Z$, see Section~\ref{sectrescur}.
    The existence of such $\xi_i$ is proved in \cite{A2}, but the tuple $\xi$ is not explicitly
    given. What we prove in Theorem~\ref{thmxivanish} is that if $R^Z$ is the current associated
    with a \emph{minimal} free resolution, then all $\xi_i$ vanish at $Z_\sing$. 
    This result can be seen as a generalization of the Poincar\'e-Lelong formula from
    the reduced complete intersection case to the Cohen-Macaulay case.
    In the reduced complete intersection case, the representation \eqref{eqintcurrepr} is given
    by the Poincar\'e-Lelong formula, and since in that case, $\xi$ is explicitly
    given, the fact that $\xi$ vanish at $Z_\sing$ follows from the implicit function
    theorem, see Section~\ref{sectrcicase}.

    Summarizing Theorem~\ref{annmuf} and Propositions \ref{propcounterex1} and \ref{propcounterex2},
    we get the following.

    \begin{cor} \label{corqduality}
        Assume that $\codim Z^k \geq k + p$ for all $k \geq 0$, with equality for some $k$.
        Then $(Z,w)$ has $q$-duality for $q < p$ and all $w$ in some neighborhood of $z$,
        and $(Z,w)$ does not have $q$-duality for $q = p$ for some $w$ arbitrarily close to $z$.
        In addition, if $\codim Z_\sing = p$, that is, we have equality for $k = 0$,
        then $(Z,w)$ does not have $q$-duality for $q > p$ for some $w$ arbitrarily close to $z$.
    \end{cor}

    \begin{proof}
        The only part that does not follow immediately from Theorem~\ref{annmuf}, Proposition~\ref{propcounterex1}
        and Proposition~\ref{propcounterex2} is if $q > p$, $(Z,z)$ is not Cohen-Macaulay but
        there is equality in $\codim Z^k \geq k + p$ for $k = 0$.
        However, in that case, $\codim Z^0 = p$ and $\codim Z^1 \geq p + 1$,
        so since $Z^0 \supseteq Z^1$, there is some $w \in Z^0$ arbitrarily close to $z$ such that
        $(Z,w)$ is Cohen-Macaulay (i.e., $w \in Z^0 \setminus Z^1$), and we can apply Proposition~\ref{propcounterex1}.
    \end{proof}

\section{The case of a reduced complete intersection} \label{sectrcicase}

We begin by showing how to prove Corollary~\ref{corqduality} in the case when $Z$ is a reduced complete
intersection, i.e., that $Z = \{ h_1 = \dots = h_r = 0 \}$, where $r = \codim Z$, and $dh_1\wedge \dots \wedge dh_r \neq 0$
generically on $Z$.
\begin{prop} \label{propciduality}
    Let $(Z,z)$ be a reduced complete intersection and assume that $\codim Z_\sing = p$.
    Then, for all $w$ in some neighborhood of $z$, $(Z,w)$ has $q$-duality for $q < p$,
    and there exists $w$ arbitrarily close to $z$ such that $(Z,w)$ does not have $q$-duality
    for $q \geq p$.
\end{prop}
In this case, the main ideas of the proof in the general case appear, but it only involves
the Coleff-Herrera product, and hence we avoid many of the technicalities of the proof in the general case.

By the Poincar\'e-Lelong formula, see Section 3.6 in \cite{CH},
\begin{equation} \label{eqpl}
    \frac{1}{(2\pi i)^r} \dbar \frac{1}{h_r}\wedge \dots \wedge \dbar \frac{1}{h_1} \wedge dh_1 \wedge \dots \wedge dh_r = [Z].
\end{equation}
Now, let $f = (f_1,\dots,f_q)$ be a complete intersection on $Z$, and consider $\mu^f$ as a current in the ambient space,
as given by \eqref{eqmufambientreg}.
By considering the regularization of $\mu^f$ in \eqref{eqmufambientreg}, using the Poincar\'e-Lelong formula \eqref{eqpl}
on $[Z]$, and also regularizing $\mu^h$ in \eqref{eqpl}, we get
\begin{equation} \label{eqmufambientplreg}
    i_* \mu^f = \left. \frac{\dbar |f_q|^{2\lambda_2}\wedge \dots \wedge \dbar |f_1|^{2\lambda_2} \wedge \dbar |h_r|^{2\lambda_1} \wedge \dots \wedge \dbar |h_1|^{2\lambda_1}}{f_q\dots f_1 h_r \dots h_1} \wedge \eta \right|_{\lambda_1 = 0, \lambda_2 = 0},
\end{equation}
where $\eta = (2\pi i)^{-r} dh_1 \wedge \dots \wedge dh_r$.
Note that $f$ being a complete intersection on $Z$ means that $(f,h)$ is a complete intersection on $\Cn$.
In this case, by results of Samuelsson, \cite{hasamJFA}, the right-hand side of \eqref{eqmufambientplreg} is continuous in
$(\lambda_1,\lambda_2)$ near $(0,0)$.
In particular, we can instead take the analytic continuation where $\lambda_1 = \lambda_2 = \lambda$ to $\lambda=0$,
which equals the Coleff-Herrera product of $(f,h)$, i.e.,
\begin{equation} \label{eqmufambientpl}
    i_* \mu^f = \dbar \frac{1}{f_q}\wedge \dots \wedge \dbar \frac{1}{f_1}\wedge \dbar \frac{1}{h_r} \wedge \dots \wedge \dbar \frac{1}{h_1}\wedge \eta.
\end{equation}

The representation \eqref{eqmufambientpl} of the Coleff-Herrera product will be the basis of proving Proposition~\ref{propciduality}.
First, we consider the case when $q < p$.
Since $Z$ is a reduced complete intersection, $\eta = (2\pi i)^{-r} dh_1 \wedge \dots \wedge dh_r$ is non-vanishing on $Z_\reg$.
Thus, if $g \in \ann \mu^f$, i.e., by considering $g$ in the ambient space, $g i_* \mu^f = 0$,
we get from \eqref{eqmufambientpl} that $g$ annihilates the Coleff-Herrera product $\mu^{(f,h)}$ on $Z_\reg$.
The Coleff-Herrera product belongs to a class of currents called pseudomeromorphic currents, see Section~\ref{sectrescur}.
This class of currents is closed under multiplication with smooth functions, and have the property that if $T$ is
a pseudomeromorphic $(*,k)$-current with support on a variety of codimension $> k$, then $T = 0$, see Proposition~\ref{proppm0}.
Thus, the current $g \mu^{(f,h)}$ is in fact $0$, since it is a $(0,q+r)$-current with support on $Z_\sing$ which has
codimension $p + r$ (in $\Cn$).
By the duality theorem (on $\Cn$), $g \in \mathcal{J}(f,h)$, i.e., $g \in \mathcal{J}(f)$ in $\Ok_Z = \Ok/\mathcal{J}(h)$.
Hence, $Z$ has $q$-duality if $q < p$.

We now consider the case when $q \geq p$. We can find $w$ arbitrarily close to $z$,
and a complete intersection $f = (f_1,\dots,f_q)$ on $(Z,w)$ such that $Z(f) \subseteq Z_\sing$,
see Section~\ref{sectci}, and in particular Lemma~\ref{lmaci3}.
Let $\mathcal{I} = \mathcal{J}(f_1,\dots,f_q)$, and $V = Z(I)$. It follows from the Nullstellensatz
that there exists a holomorphic function $g$ such that $g \notin \mathcal{I}$, but $g \in \mathcal{J}_V$
and $g \mathcal{J}_V \subseteq \mathcal{I}$, see the proof of Proposition~\ref{propcounterex1}
in Section~\ref{sectproofce1}.

We claim the $g$ annihilates $\mu^f$, and since $g \notin \mathcal{J}(f)$, this proves the second part
of Proposition~\ref{propciduality}. To prove this claim, note first that by the implicit function theorem,
$\eta = (2\pi i)^{-r} dh_1 \wedge \dots \wedge dh_r$ vanishes on $Z_\sing$, i.e., if $\eta = \sum_{|I|=r} h_I dz_I$,
then each $h_I \in \mathcal{J}_{Z_\sing}$.
Since $V \subseteq Z_\sing$, we get that $\mathcal{J}_{Z_\sing} \subseteq \mathcal{J}_V$.
Hence, $g h_I \in g \mathcal{J}_{Z_\sing} \subseteq g \mathcal{J}_V \subseteq \mathcal{I}$,
where the last inclusion follows by the choice of $g$.
Thus, we get from multiplying \eqref{eqmufambientpl} by $g$ that $g$ annihilates $\mu^f$, since
each term $g h_I$ from $g \eta$ annihilates the Coleff-Herrera product $\mu^{(f,h)}$.

\section{Residue currents and free resolutions} \label{sectrescur}

When $Z$ is a reduced complete intersection defined by $h$, the Coleff-Herrera product
$\mu^h$ is a natural current associated to $Z$, and in Section~\ref{sectrcicase}, the factorization
of the integration current $[Z]$ in terms of $\mu^h$ was the starting point of the argument.
We want to find a corresponding current $R^Z$ and a factorization of the integration current $[Z]$
also when $Z$ is not a complete intersection, see Theorem~\ref{thmxivanish} below.
To do this, we use a construction by Andersson and Wulcan of currents associated to free resolutions
of ideals, \cite{AW1}.

Let $\mathcal{J}$ be a coherent ideal sheaf, and let $(E,\varphi)$ be a locally free resolution
of the sheaf $\Ok/\mathcal{J}$ as in \eqref{eqfreeres}.
Mostly, we will use the case when $\mathcal{J} = \mathcal{J}_Z$,
the sheaf of holomorphic functions vanishing on the analytic variety $Z$.

In particular, if $Z$ is a reduced complete intersection, and $\mathcal{J}_Z = \mathcal{J}(h_1,\dots,h_p)$, then
the Koszul complex of $h$ is a free resolution of $\Ok/\mathcal{J}_Z$.
In this case, the current associated to the Koszul complex of $h$ equals the Coleff-Herrera product $\mu^h$,
Theorem~\ref{thmbmch}.

To construct the current associated to $E$, one first defines, outside of $Z = Z(\mathcal{J})$, right inverses
$\sigma_k : E_{k-1} \to E_k$ to $\varphi_k$ which are minimal with respect to some metric on $E$, i.e.,
$\varphi_k \sigma_k|_{\Im \varphi_k} = \Id_{\Im \varphi_k}$, $\sigma_k = 0$ on $(\Im \varphi_k)^\perp$,
and $\Im \sigma_k \perp \ker \varphi_k$. One lets
\begin{equation*}
    u = \sigma_1 + \sigma_2\dbar\sigma_1 + \dots + \sigma_N\dbar\sigma_{N-1}\dots\dbar\sigma_1.
\end{equation*}
Then, if $F \not\equiv 0$ is a holomorphic function vanishing at $Z$, $R^E$ is defined by
\begin{equation} \label{eqrlambda}
    R^E = \dbar |F|^{2\lambda} \wedge u |_{\lambda = 0},
\end{equation}
where for $\Re \lambda \gg 0$, this is a (current-valued) analytic function in $\lambda$,
and $|_{\lambda = 0}$ denotes the analytic continuation to $\lambda = 0$.
See \cite{AW1} for more details.

Let $R^E_k$ denote the part of $R^E$ with values in $E_k$, i.e., $R^E_k$ is
a $E_k$-valued $(0,k)$-current. If $Z = Z(\mathcal{J})$, and $\codim Z = p$,
then we will in fact have that
\begin{equation} \label{eqrepN}
    R^E = R^E_p + \dots + R^E_N,
\end{equation}
where $N$ is the length of the free resolution $(E,\varphi)$.

The current $R^E$ has the following crucial property, \cite{AW1}, Theorem~1.1.

\begin{thm} \label{thmreann}
    Let $R^E$ be the current associated to a free resolution $(E,\varphi)$
    of an ideal $\mathcal{J}$.
    Then $\ann R^E = \mathcal{J}$.
\end{thm}

If $Z$ is an analytic subvariety, we will denote by $R^Z$ the current associated
with a free resolution of $\mathcal{J}_Z$ of minimal length. Note that
this current is not in general uniquely defined, as it might depend on the choice of metrics.

In this article, we are only concerned with local (or semi-local) statements,
so the reader may very well assume the vector bundles are in fact free modules.
However, we still keep the notation of vector bundles, partly to keep a consistent notation,
but also since it is advantageous to be able to refer to the specific vector bundle $E_k$
and not just the free module $\Ok^{\oplus r_k}$.

If $f = (f_1,\dots,f_p)$ defines a complete intersection, the Coleff-Herrera product coincides with
the so called \emph{Bochner-Martinelli current} of $f$, as introduced by Passare, Tsikh and Yger in \cite{PTY}
in the smooth case. It was also developed in the case of an analytic variety in \cite{BVY}.
If $f$ defines a complete intersection, the Bochner-Martinelli current of $f$, denoted $R^f$,
can be defined as the current associated with the Koszul complex of $f$.
In fact, in \cite{AW1}, currents associated with any generically exact complex of vector bundles
are defined, and not only free resolutions as described above, and then the Bochner-Martinelli
current for an arbitrary $f$ can be defined as the current associated with the Koszul complex of $f$,
see \cite{A3}.
This equality of the Coleff-Herrera product and the Bochner-Martinelli current makes the Coleff-Herrera
product fit in the framework of residue currents associated with a free resolution, and this
substitution will be used throughout the arguments.
The theorem below is Theorem~4.1 in \cite{PTY} in the smooth case, and Theorem~6.3 in \cite{L}
in the singular case.

\begin{thm} \label{thmbmch}
    If $f = (f_1,\dots,f_p)$ defines a complete intersection on $Z$, then 
    the Bochner-Martinelli current $R^f$ of $f$ equals the Coleff-Herrera product
    $\mu^f$ of $f$.
\end{thm}

Pseudomeromorphic currents were introduced in \cite{AW2}.
A current of the form
\begin{equation*}
    \frac{1}{z_{i_1}^{k_1}}\cdots\frac{1}{z_{i_m}^{k_m}}\dbar \frac{1}{z_{i_{m+1}}^{k_{m+1}}}\wedge
    \cdots \wedge \dbar \frac{1}{z_{i_p}^{k_p}} \wedge \alpha,
\end{equation*}
where $\alpha$ is a smooth form with compact support, is called an elementary current.
A current $T$ is said to be a \emph{pseudomeromorphic current}, denoted $T \in \PM$, if it is a locally finite
sum of push-forwards of elementary currents under compositions of smooth modifications and open inclusions.
As can be seen from their construction, the Coleff-Herrera product $\mu^f$ and the current $R^E$
associated with a free resolution are pseudomeromorphic. 
We will need the following two properties of pseudomeromorphic currents, see Proposition~2.3 and Corollary~2.4 in \cite{AW2}.

\begin{prop} \label{proppm0}
    If $T \in \PM$ is of bidegree $(0,p)$ and $T$ has support on a variety
    of codimension $\geq p+1$, then $T = 0$.
\end{prop}

\begin{prop} \label{proppm1}
    If $T \in \PM$ has support on $Z$, and if $f$ is a holomorphic function
    vanishing on $Z$, then $\bar{f} T = 0$.
\end{prop}

We will use results from \cite{A1}, that one can define products of the currents $R^f$ and $R^Z$,
and that under certain conditions, the annihilator of the product $R^f \wedge R^Z$
equals the sum of the ideals $\mathcal{J}(f) + \mathcal{J}_Z$.
This type of product can be defined more generally for currents $R^E$ and $R^F$
associated with two free resolutions $E$ and $F$. If $R^E$ is defined by
\begin{equation*}
    R^E := \dbar |G|^{2\lambda} \wedge u |_{\lambda = 0},
\end{equation*}
then $R^E \wedge R^F$ can be defined by
\begin{equation*}
    R^E \wedge R^F := \dbar |G|^{2\lambda} \wedge u \wedge R^F |_{\lambda = 0}.
\end{equation*}
\begin{remark}
    If we consider $R^f \wedge R^Z$, where $f = (f_1,\dots,f_p)$ is a strongly holomorphic mapping on $Z$,
    then this depends a priori on the choice of representatives of $f$ in the ambient space.
    We will only need that under certain conditions,
    $\ann R^f \wedge R^Z = \mathcal{J}(f) + \mathcal{J}_Z$, which is independent of the choice of representatives.
    However, one can in fact show that $R^f \wedge R^Z$ does not depend on the choice of representatives,
    essentially due to that $R^Z$ is annihilated by both holomorphic and anti-holomorphic functions vanishing on $Z$.
\end{remark}

If
\begin{equation*}
    0 \to E_n \xrightarrow{\varphi_n} E_{n-1} \to \cdots \xrightarrow{\varphi_1} E_0 \to 0
\end{equation*}
and
\begin{equation*}
    0 \to F_m \xrightarrow{\psi_m} F_{m-1} \to \dots \xrightarrow{\psi_1} F_0 \to 0
\end{equation*}
are two complexes, then one can form the tensor product of the complexes, denoted $(E\otimes F,\varphi\otimes \psi)$,
by letting $ (E\otimes F)_k = \oplus_{i + j = k} E_i \otimes F_j$ and
$ (\varphi\otimes \psi) (\xi \otimes \eta) = \varphi_i \xi \otimes \eta + (-1)^i \xi \otimes \psi_j \eta$
if $\xi \otimes \eta \in E_i \otimes F_j$.

The following theorem, Theorem~4.1 and Remark~8 in \cite{A1}, and its corollary gives conditions for when the annihilator of
$R^E \wedge R^F$ coincides with the sum of the annihilators, and when the tensor product
of two (minimal) free resolutions is a (minimal) free resolution.

\begin{thm} \label{thmanntensprod}
    Let $(E,\varphi)$ and $(F,\psi)$ be free resolutions of ideal sheaves
    $\mathcal{I}$ and $\mathcal{J}$, and let $Z^\mathcal{I}_k$ and $Z^\mathcal{J}_l$ be the associated
    sets where $\varphi_k$ and $\psi_l$ does not have optimal rank.
    If $\codim ( Z^\mathcal{I}_k \cap Z^\mathcal{J}_l ) \geq k + l$ for all $k,l \geq 1$, then
    $\ann R^E \wedge R^F = \mathcal{I} + \mathcal{J}$ and $(E \otimes F,\varphi\otimes \psi)$ is a free
    resolution of $\mathcal{I} + \mathcal{J}$. In addition, if both $E$ and $F$ are minimal free resolutions
    at some point $z$, then the tensor product is a minimal free resolution.
\end{thm}

    To be precise, the last statement is not included in \cite{A1}.
    However, if the tensor product is a free resolution, it follows immediately
    from the definition of minimality at some $z$, that $\Im \varphi_k \subseteq \mathfrak{m}_z \Ok(E_{k-1})$
    (where $\mathfrak{m}_z$ denotes the maximal ideal of $\Ok_{\Cn,z}$), that it is minimal.

\begin{cor} \label{corranntensprod}
    If $f = (f_1,\dots,f_p)$ is a reduced complete intersection on $Z$, and $\codim Z_f \cap Z^l \geq p + l$
    for $l \geq 1$, then $\ann R^f \wedge R^Z = \mathcal{J}(f) + \mathcal{J}_Z$, and the tensor product of the Koszul
    complex of $f$ and a free resolution of $\mathcal{J}_Z$ is a free resolution of $\mathcal{J}(f) + \mathcal{J}_Z$.
    In addition, if the free resolution of $\mathcal{J}_Z$ is minimal at some point $z$, 
    then the tensor product is a minimal free resolution.
\end{cor}

\begin{proof}
    If $f$ is a complete intersection, then the Koszul complex of $f$ is a minimal free resolution,
    and its associated singularity subvarieties $Z^f_k$ are equal to $Z_f$ for $k \leq p$,
    and empty for $k > p$.
    Since $Z_l = Z$ for $l \leq \codim Z$, the condition $\codim Z_f \cap Z_l \geq p + l$ is
    automatic for $l \leq \codim Z$ since $f$ is a complete intersection on $Z$.
    Thus, the condition $\codim Z^f_k \cap Z_l \geq k + l$ becomes just
    $\codim Z_f \cap Z^l \geq p + l$.
\end{proof}

\section{Proof of Theorem~\ref{annmuf}} \label{sectproofannmuf}

    The inclusion $\mathcal{J}(f_1,\dots,f_p) \subseteq \ann \mu^f$ follows from Theorem~\ref{thmanninclusion}
    (also without the conditions on $Z^k \cap Z_f$), so we only need to prove the reverse inclusion.
    Assume that $Z \subseteq \Omega \subseteq \Cn$ and that $\codim Z = q$.
    Then $i_* \mu^f = \mu^f \wedge [Z]$, where $i : Z \to \Omega$ is the inclusion,
    and by Theorem~\ref{thmbmch}, $\mu^f \wedge [Z] = R^f \wedge [Z]$.
    We will show that $g \in \ann (R^f\wedge [Z])$ implies that $g \in \ann (R^f \wedge R^Z)$
    (which does not hold in general, but does under the conditions of the theorem).
    By $(3.21)$ in \cite{AS}, outside of $Z_\sing$ there exists a smooth $(q,0)$-vector field $\gamma$
    such that $\gamma \imult [Z] = R^Z_q$.
    Then, outside of $Z_\sing$,
    \begin{equation*}
        g R^f \wedge R^Z_q = g R^f \wedge (\gamma \imult [Z]) = \gamma \imult (g R^f \wedge [Z]) =  0.
    \end{equation*}
    Hence $g R^f \wedge R^Z_q$ is a $(0,p + q)$-current with support on $Z_f \cap Z_\sing$,
    so by Proposition~\ref{proppm0}, it is $0$ since $Z_f \cap Z_\sing$ has codimension $\geq p + q + 1$.

    Outside of $Z^{k+1}$, there exists a smooth $\Hom(E_{q + k},E_{q + k + 1})$-valued smooth $(0,1)$-form
    $\alpha_{q + k + 1}$ such that $R^Z_{q + k + 1} = \alpha_{q + k + 1}R^Z_{q + k}$,
    see \cite{AW1}.
    We will prove by induction that
    \begin{equation} \label{ind}
        g R^f \wedge R^Z_{q + k} = 0.
    \end{equation}
    Above we proved this for $k = 0$, so let us assume that it is proved for $k$. Then, outside of $Z^f \cap Z^{k+1}$,
    \begin{equation*}
    g R^f \wedge R^Z_{q + k + 1} = \alpha_{q + k + 1}(g R^f \wedge R^Z_{q + k}) = 0.
    \end{equation*}
    Thus $g R^f \wedge R^Z_{q + k + 1}$ has support on $Z^f \cap Z^{k + 1}$
    which has codimension $\geq p + q + k + 2$, and since it is a pseudomeromorphic current of bidegree $(0,p + q + k + 1)$,
    it is $0$ by Proposition~\ref{proppm0}. Thus we have proven that $g \in \ann(R^f\wedge R^Z)$.
    By Corollary~\ref{corranntensprod}, $\ann (R^f \wedge R^Z) = \mathcal{J}(f) + \mathcal{J}_Z$,
    and hence we get that $g \in \mathcal{J}(f) + \mathcal{J}_Z$.

\section{Complete intersections and choice of coordinates} \label{sectci}

This section contains several lemmas about choices of coordinates and existence
of complete intersections containing a certain variety.
They will be used throughout the rest of the sections.
This first lemma, which is based on the first lemma in Section~5.2.2 in \cite{GrRe},
is the basis for the rest of them.

\begin{lma} \label{lmaci1}
    Assume that $(V,z) \subseteq (Z,z)$, where $(Z,z)$ has pure dimension,
    $V$ has codimension $\geq 1$ in $Z$ and that there exists $f = (f_1,\dots,f_m)$
    such that $(V,z) = (Z,z) \cap \{ f_1 = \dots = f_m = 0 \}$.
    Then there exists a finite union, $E$, of proper linear subspaces of $\C^m$,
    such that $(Z,z) \cap \{ a \cdot f = 0 \}$ has codimension $1$ in $(Z,z)$
    if $a \in \C^m \setminus E$.
\end{lma}

\begin{proof}
    The set $E$ of $a \in \C^m$ such that $(Z,z) \cap \{ a \cdot f = 0 \} = (Z,z)$
    is a linear subspace of $\C^m$, and since $(Z,z) \cap \{ f_1 = \dots = f_m = 0 \}$
    has positive codimension, it must be a proper subspace.
    If $(Z,z)$ is irreducible, there thus exists a proper subspace $E \subseteq \C^m$
    such that $(Z,z) \cap \{ a \cdot f  = 0 \}$ has codimension $1$ in $(Z,z)$ if $a \in \C^m \setminus E$.
    If $(Z,z)$ is reducible, then there exists such subspaces $E_i$ for each irreducible component
    $(Z_i,z)$ of $(Z,z)$, and thus we can take $E = \cup E_i$.
\end{proof}

The following two lemmas are about existence of certain complete intersections containing
a given variety, and their existence are the basis for the counterexamples to
the duality theorem.

\begin{lma} \label{lmaci2}
    Assume that $(V,z) \subseteq (Z,z)$, where $(Z,z)$ has pure dimension,
    $\codim V = p$ in $Z$, and let $f = (f_1,\dots,f_m)$ be such that 
    $(V,z) = (Z,z) \cap \{ f_1 = \dots = f_m = 0 \}$.
    Then there exists $f' = (f_1',\dots,f_p')$, a complete intersection on $Z$,
    such that $(V,z) \subseteq (V',z) := (Z,z) \cap \{ f_1' = \dots = f_p' = 0 \}$,
    where $f_i' = \sum a_{i,j} f_j$.
\end{lma}

\begin{proof}
    By Lemma~\ref{lmaci1}, there exists $E \subseteq \C^m$
    such that $(Z,z) \cap \{ a \cdot f = 0 \}$ has codimension $1$ in $(Z,z)$ for $a \in \C^m \setminus E$.
    We choose $f_1' = a \cdot f$, for some $a \in \C^m \setminus E$.
    Proceeding in the same way with $(Z,z) \cap \{ f_1' = 0 \}$ instead of $(Z,z)$, we get $f_2'$ such that
    $(Z,z) \cap \{ f_1' = f_2' = 0 \}$ has codimension $2$ in $Z$.
    Repeating this, $f' = (f_1',\dots,f_p')$ will be the desired complete intersection.
\end{proof}

\begin{lma} \label{lmaci3}
    Assume that $(V,z) \subseteq (Z,z)$, where $(V,z)$ has codimension $p$ in $(Z,z)$
    and $\dim (Z,z) = d$.
    Then, for some $w$ arbitrarily close to $z$, there exists a complete intersection
    $f = (f_1,\dots,f_d) \in \Ok_{Z,w}^{\oplus d}$ such that
    $(V,w) = (Z,w) \cap \{ f_1 = \dots = f_p = 0 \}$.
\end{lma}

\begin{proof}
    By Lemma~\ref{lmaci2}, there exists $f = (f_1,\dots,f_p)$ a complete intersection
    on $(Z,z)$ such that $(V,z) \subseteq (V',z)$, where $V' = \{ f_1=\dots=f_p = 0 \}$.
    Since the set where $V'$ is reducible has codimension $> p$,
    there exists some $w$ arbitrarily close to $z$ such that $(V,w) = (V',w)$.
    Then we apply Lemma~\ref{lmaci2} again to $(\{ w \},w) \subseteq (V,w)$ to find
    $(f_{p+1},\dots,f_d)$, a complete intersection on $(V,w)$, so that 
    $f = (f_1,\dots,f_d)$ is the desired complete intersection.
\end{proof}

This last lemma is about the existence of a certain choice of coordinates, which is used
in the proof of Theorem~\ref{thmxivanish}.

\begin{lma} \label{lmaci4}
    Let $(Z,0) \subseteq (\Cn,0)$ and assume that $Z$ has pure dimension $d$.
    Then we can choose coordinates $w$ on $\Cn$ such that $(Z,0) \cap \{ w_I = 0 \} = \{ 0 \}$
    for all $I \subseteq \{ 1,\dots, n \}$ with $|I| = d$.
\end{lma}

\begin{proof}
    We will choose the coordinates $w$ on $\Cn$ inductively.
    By \linebreak Lemma~\ref{lmaci1}, there exists $E$ such that $(Z,0) \cap \{ a \cdot z = 0 \}$
    has codimension $1$ in $Z$ if $a \notin E$, and we choose $w_1 = a \cdot z$
    for some $a \notin E$.
    Now, we assume by induction that we have chosen coordinates $(w_1,\dots,w_k)$
    such that $(Z,0) \cap \{ w_I = 0 \}$ has codimension $|I|$ for each $I \subseteq \{ 1,\dots, k\}$
    with $|I| \leq d$.
    For each $I \subseteq \{1,\dots,k\}$ with $|I| \leq d-1$, we can then find $E_I$ by Lemma~\ref{lmaci1}
    such that $(Z,0) \cap \{ w_I = 0 \} \cap \{ a \cdot z = 0 \}$ has codimension $1$ in $(Z,0) \cap \{ w_I = 0 \}$
    if $a \notin E_I$. Since each $E_I$ is a finite union of proper subspaces of $\Cn$, we can find
    $a \in \Cn \setminus \cup E_I$, and we then let $w_{k+1} = a \cdot z$.
    Proceeding in this way, $w = (w_1,\dots,w_n)$ will be the desired choice of coordinates.
\end{proof}

\section{Representations of the integration current in the Cohen-Macaulay case} \label{sectreprintcurrent}

To prove Proposition~\ref{propcounterex1}, we will use
the following representation of the integration current $[Z]$ on $Z$
in terms of the current $R^Z$.
Assume that $Z$ is Cohen-Macaulay, and that $\codim Z = p$, so that $R^Z = R^Z_p$ by \eqref{eqrepN}.
By Example 1, \cite{A2}, there exist holomorphic $(p,0)$-forms $\xi_i$ such that
\begin{equation} \label{eqZrepr1}
    [Z] = \sum \xi_i \wedge R^Z_{p,i},
\end{equation}
where $R^Z_{p,i}$ are the various components of $R^Z$, i.e., given a local frame $(e_1,\dots,e_N)$
of $\Ok(E_p)$, $R^Z_p = \sum R^Z_{p,i} e_i$.

If $Z$ is a reduced complete intersection defined by $f = (f_1,\dots,f_p)$, 
then $R^Z = \mu^f$ by Theorem~\ref{thmbmch}, and by the Poincar\'e-Lelong formula, see \cite{CH},
we have
\begin{equation*}
    [Z] = \frac{1}{(2\pi i)^p}\dbar\frac{1}{f_p}\wedge \dots \wedge \dbar \frac{1}{f_1} \wedge df_1\wedge \dots \wedge df_p.
\end{equation*}
Thus, we can take $\xi = df_1 \wedge \dots \wedge df_p$,
and then it is clear by the implicit function theorem that $\xi$ vanishes at $Z_\sing$.
We will show that this is the case also when $Z$ is Cohen-Macaulay.
This is Theorem~\ref{thmxivanish}, and the proof will use the following lemmas.
Recall that the \emph{socle} of module $M$ over a local ring $(R,\mathfrak{m},k)$
is defined as $\Hom_R(k,M)$, see \cite{BH}.
We will use the following characterization of the socle, which is immediate
from the definition:
\begin{equation} \label{eqhomomoq}
    \Hom_R(k,M) \cong \{ \alpha \in M \ |\ \mathfrak{m} \alpha = 0 \}.
\end{equation}

\begin{lma} \label{lmadimsoclerank}
    Let $\mathfrak{q}$ be a germ of an ideal at $0$ such that $\sqrt{\mathfrak{q}} = \mathfrak{m}$, where $\mathfrak{m}$
    is the maximal ideal at $0$, and let
    \begin{equation} \label{eqfreeresq}
        0 \to \Ok(E_n) \xrightarrow{\varphi_n} \dots \xrightarrow{\varphi_1} \Ok(E_0) \to \Ok/\mathfrak{q} \to 0
    \end{equation}
    be a minimal free resolution of $\Ok/\mathfrak{q}$, where $\Ok = \Ok_{\Cn,0}$. Then
    \begin{equation*}
        \dim_\C \Hom_\Ok(\Ok/\mathfrak{m},\Ok/\mathfrak{q}) = \rank E_n.
    \end{equation*}
\end{lma}

\begin{proof}
    We have
    \begin{equation*}
        \rank E_n = \dim \Tor_n(\Ok/\mathfrak{m},\Ok/\mathfrak{q})
    \end{equation*}
    since $\Tor_n(\Ok/\mathfrak{m},\Ok/\mathfrak{q})$ is just the $n$:th homology of the complex \eqref{eqfreeresq}
    tensored with $\Ok/\mathfrak{m}$. This is $\C^{\rank E_n}$ since the free resolution is minimal so that if
    \begin{equation*}
        \tilde{\varphi}_n : \Ok(E_n) \otimes \Ok/\mathfrak{m} \to \Ok(E_{n-1}) \otimes \Ok/\mathfrak{m},
    \end{equation*}
    then $\tilde{\varphi}_n = 0$ since $\Im \varphi_n \subseteq \mathfrak{m} E_{n-1}$ by definition of minimality of a free resolution.
    However, $\Tor_n(\Ok/\mathfrak{m},\Ok/\mathfrak{q})$
    can also be computed by taking a free resolution of $\Ok/\mathfrak{m}$, tensoring it with $\Ok/\mathfrak{q}$ and 
    taking homology. Since the Koszul complex of $(z_1,\dots,z_n)$ is a free resolution of $\Ok/\mathfrak{m}$, we get
    \begin{align*}
        \Tor_n(\Ok/\mathfrak{m},\Ok/\mathfrak{q}) \cong \Ker\left(\bigwedge^n \Ok/\mathfrak{q} \xrightarrow{\delta_z} \bigwedge^{n-1} \Ok/\mathfrak{q}\right) \\
        \cong \{ \alpha \in \Ok/\mathfrak{q}\ |\ \mathfrak{m}\alpha = 0 \} \cong \Hom_\Ok(\Ok/\mathfrak{m},\Ok/\mathfrak{q}),
    \end{align*}
    where the last equality is \eqref{eqhomomoq}.
\end{proof}

\begin{lma} \label{lmagensofideal}
    Assume that there exist pseudomeromorphic currents $\mu_1,\dots,\mu_N$ such that $\mathfrak{q} = \cap \ann \mu_i$,
    where $\mathfrak{q}$ is an ideal such that $\sqrt{\mathfrak{q}} = \mathfrak{m}$.
    Then
    \begin{equation*}
        N \geq \dim_\C \Hom_\Ok(\Ok/\mathfrak{m},\Ok/\mathfrak{q}).
    \end{equation*}
\end{lma}

\begin{proof}
    We claim that there exists a $\C$-linear injective mapping
    \begin{equation*}
        \tilde{\mu} : \Hom_\Ok(\Ok/\mathfrak{m},\Ok/\mathfrak{q}) \to \C^N,
    \end{equation*}
    which proves the statement.
    We consider $\Hom_\Ok(\Ok/\mathfrak{m},\Ok/\mathfrak{q})$ as \eqref{eqhomomoq}.
    Since $\mathfrak{q} \subseteq \ann \mu_i$, the mapping $\alpha \mapsto \alpha \mu_i, \alpha \in \Hom_\Ok(\Ok/\mathfrak{m},\Ok/\mathfrak{q})$
    is well-defined. Since $\mathfrak{m}\alpha = 0$, and $\overline{\mathfrak{m}}\mu_i = 0$ by Proposition~\ref{proppm1},
    $\alpha\mu_i$ is a current of order $0$ with support on $\{ 0 \}$. Thus
    \begin{equation} \label{eqphimui}
        \alpha \mu_i = a_i R_0,
    \end{equation}
    for some $a_i \in \C$, where $R_0$ is the current $\delta_{z = 0} d\bar{z}$,
    that is, $R_0 . \alpha dz = \alpha(0)$.
    We thus get a mapping
    \begin{equation*}
        \tilde{\mu}(\alpha) = (a_1,\dots,a_N),
    \end{equation*}
    where $a_i$ are defined by \eqref{eqphimui}.
    It only remains to see that $\tilde{\mu}$ is injective. However, if $\tilde{\mu}(\alpha) = 0$,
    then $\alpha \in \cap \ann \mu_i = \mathfrak{q}$, so $\alpha = 0$ in $\Ok/\mathfrak{q}$.
\end{proof}

Combining Lemma~\ref{lmadimsoclerank} and Lemma~\ref{lmagensofideal}, if $f$ is a complete intersection
on $Z$, where $Z$ is Cohen-Macaulay, then none of the components in the decomposition $R^f \wedge R^Z = \sum R^f \wedge R^Z_{p,i}$
are redundant. This will be a crucial step in the proof of the following theorem.

\begin{thm} \label{thmxivanish}
    Let $Z \subseteq \Omega \subseteq \Cn$ be a subvariety of $\Omega$ of codimension $p$, and assume that $Z$ is
    Cohen-Macaulay. Then there exists holomorphic $(p,0)$-forms $\xi_i$ such that
    \begin{equation*}
        [Z] = \sum \xi_i \wedge R^Z_{p,i},
    \end{equation*}
    and if $R^Z$ is defined with respect to a minimal free resolution of $\Ok_Z$, then 
    all $\xi_i$ vanish at $Z_\sing$.
\end{thm}

\begin{proof}
    As mentioned in the introduction of the section, the existence of $\xi_i$ is
    Example~1 in \cite{A2}, so we only need to prove that $\xi_i$ vanish at $Z_\sing$
    if $R^Z$ is defined with respect to a minimal free resolution.
    Assume that $0 \in Z_\sing$.
    We begin by choosing coordinates in $\Cn$ such that $\{ w_J = 0 \} \cap Z = \{ 0 \}$
    for all $J \subseteq \{ 1,\dots,n \}$ with $|J| = n-p$, which is possible by Lemma~\ref{lmaci4}.
    We have
    \begin{equation} \label{eqZrepr2}
        [Z] = \sum_{i,|I| = p} \xi_{I,i} dw_I \wedge R^Z_{p,i},
    \end{equation}
    where $\xi_{I,i}$ are holomorphic functions, and we are done if we can prove that $\xi_{I,i}(0) = 0$ for all $\xi_{I,i}$.

    Fix some $I \subseteq \{ 1,\dots,n \}$ with $|I| = p$. Let $w' = (w_{J_1},\dots,w_{J_{n-p}})$, where $J = I^c$.
    By the Poincar\'e-Lelong formula applied to $w'$ on $Z$, see \cite{CH}, Section~1.9, we have that
    \begin{equation*}
        \frac{1}{(2\pi i)^p} R^{w'} \wedge dw' \wedge [Z] = k[0]
    \end{equation*}
    for some $k \geq 1$. Combined with the Poincar\'e-Lelong formula applied to $w$ in $\Cn$,
    we get
    \begin{equation*}
        R^w \wedge dw = (2\pi i)^n [0] = ((2\pi i)^{n-p}/k) R^{w'}\wedge dw' \wedge [Z].
    \end{equation*}
    Since by \eqref{eqZrepr2}
    \begin{equation*}
        dw' \wedge [Z] = \pm \sum_i \xi_{I,i} dw \wedge R^Z_{p,i}
    \end{equation*}
    we get that
    \begin{equation} \label{eqmuw}
        R^w = C \sum_i \xi_{I,i} R^{w'} \wedge R^Z_{p,i}
    \end{equation}
    for some constant $C \neq 0$.

    We first consider the case when $R^Z$ consists of one single component $R^Z_p$.
    By Corollary~\ref{corranntensprod}, $\ann (R^{w'} \wedge R^Z_p) = \mathcal{J}(w') + \mathcal{J}_Z$.
    We claim that the inclusion $\mathcal{J}(w)_0 \supseteq ( \mathcal{J}(w') + \mathcal{J}_Z )_0$ is strict.
    If the inclusion is not strict, then $w'$ generates the maximal ideal $\mathfrak{m}_{Z,0}$ in $\Ok_{Z,0}$,
    which is a contradiction by Proposition~4.32 in \cite{Dem}, since the number of functions
    needed to generate the maximal ideal at a singular point must be strictly larger than the dimension.
    Thus there exists a $g$ in 
    \begin{equation*}
        \mathcal{J}(w)_0 \setminus ( \mathcal{J}(w') + \mathcal{J}_Z)_0 = (\ann R^w)_0 \setminus (\ann (R^{w'} \wedge R^Z_p))_0.
    \end{equation*}
    Multiplying \eqref{eqmuw} by $g$, we get that $g\xi_I \in \ann (R^{w'} \wedge R^Z_p)$, and hence we must have $\xi_I(0) = 0$.

    Now we consider the case when $R^Z_p$ consists of more than one component.
    By Corollary~\ref{corranntensprod}, the tensor product of the Koszul complex of $w'$ and the minimal free
    resolution of $\mathcal{J}_Z$ is a minimal free resolution of $\mathfrak{q} := \mathcal{J}(w') + \mathcal{J}_Z$,
    and the rank $N$ of its left-most non-zero module is equal to the rank of the left-most non-zero module
    in the free resolution of $\mathcal{J}_Z$ since the left-most non-zero module of the Koszul complex has rank $1$.
    By Corollary~\ref{corranntensprod}, we have
    \begin{equation} \label{eqqrepr}
        \mathfrak{q} = \cap_{i=1}^N \ann (R^{w'} \wedge R^Z_{p,i}).
    \end{equation}
    By Lemma~\ref{lmadimsoclerank}, $N = \dim_\C \Hom_\Ok(\Ok/\mathfrak{m},\Ok/\mathfrak{q})$
    and by Lemma~\ref{lmagensofideal}, if $\mathfrak{q} = \cap_{i=1}^m \ann \mu_i$, then $m \geq N$.
    Thus, if we remove one term $\ann (R^{w'}\wedge R^Z_{p,i})$ from the intersection in \eqref{eqqrepr},
    we get something strictly larger, i.e., for any $i$,
    \begin{equation} \label{eqexistannihilator}
        (\cap_{j\neq i} \ann (R^{w'}\wedge R^Z_{p,j}) ) \setminus (\ann R^{w'} \wedge R^Z_{p,i}) \neq \emptyset.
    \end{equation}
    We fix some $i = 1,\dots,n$, and take $g_i$ in \eqref{eqexistannihilator} and multiply \eqref{eqmuw} by $g_i$.
    Since $g_i \in \cap_{j \neq i} \ann (R^{w'} \wedge R^Z_{p,j})$, we must have $g_i \in \mathfrak{m}$,
    so $g_i R^w = 0$. Thus we get
    \begin{equation*}
        g_i \xi_{I,i} R^{w'} \wedge R^Z_{p,i} = 0.
    \end{equation*}
    Since $g_i \notin \ann (R^{w'} \wedge R^Z_{p,i})$ but $g_i\xi_{I,i} \in \ann (R^{w'}\wedge R^Z_{p,i})$,
    we must have $\xi_{I,i} \in \mathfrak{m}$, and we are done.
\end{proof}

\section{Proof of Proposition~\ref{propcounterex1}} \label{sectproofce1}

   By moving to a nearby germ $(Z,w)$, we can assume that $Z_\sing$ has pure codimension $k$,
   and that there exists a complete intersection $f = (f_1,\dots,f_q)$ on $(Z,w)$ such that
   $(Z_\sing,w) = \{ f_1 = \dots = f_k = 0 \} \cap (Z,w)$, see Lemma~\ref{lmaci3}.
   We let $\mathcal{I} = \mathcal{J}(f_1,\dots,f_q)_w$ and $V = Z(\mathcal{I})$, and since $q \geq k$, $V \subseteq Z_\sing$.
   Since $\mathcal{J}_{V,w}$ is finitely generated over $\Ok_{Z,w}$, we get from the Nullstellensatz that
   $\mathcal{J}_{V,w}^m \subseteq \mathcal{I}$ for $m$ sufficiently large.
   Now, we choose $m$ to be minimal such that this inclusion holds. Thus, there exists a function
   $g \in \mathcal{J}_{V,w}^{m-1} \setminus \mathcal{I}$, such that $g\mathcal{J}_{V,w} \subseteq \mathcal{I}$.
   Since $g \notin \mathcal{I}$, we are done if we can show that $g\mu^{f} \wedge [Z] = 0$.

   By Theorem~\ref{thmbmch}, we can replace $\mu^f$ by $R^f$, and instead show that $g R^f \wedge [Z] = 0$.
   By Theorem~\ref{thmxivanish}
   \begin{equation*}
       g R^f \wedge [Z] = g \sum \xi_i \wedge R^f \wedge R^p_i,
   \end{equation*}
   where $\xi_i$ are holomorphic $(p,0)$-forms vanishing on $Z_\sing$.
   Thus $\xi_i = \sum \xi_{I,i} dw_I$, where $\xi_{I,i}$ are holomorphic functions
   vanishing at $Z_\sing$. Since $g \mathcal{J}_{V,w} \subseteq \mathcal{I}$ and $\mathcal{J}_{Z_\sing,w} \subseteq \mathcal{J}_{V,w}$,
   we get that $g \xi_{I,i} \in \mathcal{I}$ in $\Ok_{Z,w}$. By Corollary~\ref{corranntensprod},
   $\ann R^f \wedge R^Z = \mathcal{I} + \mathcal{J}_{Z,w}$.
   Since if $g \xi_{I,i} \in \mathcal{I}$ in $\Ok_{Z,w}$, then $g \xi_{I,i} \in \mathcal{I} + \mathcal{J}_{Z,w}$ in $\Ok_{\Cn,w}$,
   we get that $g R^f \wedge [Z] = 0$.

\section{Singularity subvarieties and counterexamples in the non Cohen-Macaulay case} \label{sectproofce2}

We will recall the notion of singularity subvarieties of analytic sheaves from \cite{ST}.
Let $R$ be a local Noetherian ring and $M \neq 0$ a finitely generated $R$-module.
A \emph{regular $M$-sequence} in an ideal $I \subseteq R$ is a sequence $(f_1,\dots,f_p)$ in $I$ such that
$f_i$ is not a zero-divisor in $M/(f_1,\dots,f_{i-1})M$ for $i = 1,\dots,p$.
The \emph{depth} of an ideal $I$ on a module $M$, denoted $\depth_I M$ is the maximal length of a regular $M$-sequence
in $I$. By $\depth_R M$, we will denote the depth of the maximal ideal $\mathfrak{m}$ of $R$ on $M$. This is also called
the homological codimension of $M$.
The \emph{homological dimension} of $M$, denoted $\hd_R M$, is defined
as the minimal length of any free resolution of $M$.

A \emph{regular local ring} is a local ring $R$ such that the maximal ideal $\mathfrak{m}$ of $R$ is generated by
$n = \dim R$ elements, where $\dim R$ is the Krull-dimension of $R$, that is, the maximal length of
a strict chain of prime ideals in $R$.
In particular, if $Z$ is an analytic variety, then $\Ok_{Z,z}$ is a regular local ring if and only
if $z \in Z_\reg$, see Proposition~4.32 in \cite{Dem}.
The following is Theorem~19.9 in \cite{E}, the Auslander-Buchsbaum formula.

\begin{prop} \label{propsyzygy}
    If $R$ is a regular local ring, and $M$ is a finitely generated $R$-module, then
    $\hd_R M + \depth_R M = \dim R$.
\end{prop}

Let $\mathcal{F}$ be a coherent analytic sheaf on $\Omega \subseteq \Cn$,
and let $\Ok_z$ denote the ring of germs of holomorphic functions at $z$ in $\Omega$.
The \emph{singularity subvarieties}, $S_m$, of $\mathcal{F}$ are defined by
\begin{equation*}
    S_m(\mathcal{F}) = \{ z \in \Omega ; \depth_{\Ok_z} \mathcal{F}_z \leq m\},
\end{equation*}
where we use the convention that $\depth_R M = \infty$ if $M = 0$, so that $S_m \subseteq \supp \mathcal{F}$.
We will use the following alternative definition of the sets $Z_k$ associated with
an analytic sheaf above:
\begin{equation} \label{eqzkalternative}
    Z_k(\mathcal{F}) = \{ z \in \Omega ; \hd_{\Ok_z} \mathcal{F}_z \geq k \}
\end{equation}
(in the introduction, we defined the sets $Z_k$ if $\mathcal{F}$ was of the form $\Ok/\mathcal{J}$,
where $\mathcal{J}$ was an coherent ideal sheaf, but the same definition works for any coherent
analytic sheaf).
To see this, note first that if $\rank \varphi_k(z)$ is constant in a neighborhood of some
$z_0 \in \Omega$ (i.e., $z_0 \notin Z_k$), then $\Ok(E_{k-1})/\Im \varphi_k$ is free in a neighborhood of $z_0$,
so $\Ok/\mathcal{J}$ has a free resolution of length $k-1$.
Conversely, by the uniqueness of minimal free resolutions, $\rank \varphi_k(z)$
must be constant in a neighborhood of $z$ if $k > \hd_{\Ok_z} \mathcal{F}_z$.

\begin{prop} \label{propzksk}
    If $\mathcal{F}$ is coherent analytic sheaf on some open set in $\Cn$, we have $S_k(\mathcal{F}) = Z_{n-k}(\mathcal{F})$.
\end{prop}

\begin{proof}
    This follows from Proposition~\ref{propsyzygy} and \eqref{eqzkalternative}.
\end{proof}

Let $\Omega \subseteq \Cn$ be an open set, $A$ a subvariety of $\Omega$ with ideal sheaf $\mathcal{J}_A$,
and $\mathcal{F}$ a coherent analytic sheaf in $\Omega$. For $z \in \Omega$, we define
\begin{equation*}
    \depth_{A,z} \mathcal{F} = \left\{ \begin{array}{cc} \infty & \text{ if } \mathcal{F}_z = 0 \\ \depth_{\mathcal{J}_{A,z}} \mathcal{F} & \text{ otherwise} \end{array} \right. .
\end{equation*}
and
\begin{equation*}
    \depth_A \mathcal{F} = \inf_{z \in A} \depth_{A,z} \mathcal{F}
\end{equation*}

The following is (part of) Theorem~1.14 in \cite{ST}.

\begin{thm} \label{thmprofsk}
    Let $\Omega \subseteq \Cn$ be some open set, $A$ a subvariety of $\Omega$, and $\mathcal{F}$ a coherent analytic sheaf in $\Omega$.
    Then for $q \geq 1$, we have $\depth_A \mathcal{F} \geq q$ if and only if
    $\dim A\cap S_{k + q}(\mathcal{F}) \leq k$ for all $k$.
\end{thm}

In particular, if we let $Z$ be an analytic subvariety of $\Omega$, $\mathcal{F} = \Ok_Z$, and $A = Z^1$, 
where the sets $Z^k$ associated with $Z$ are defined as in \eqref{eqzkintrinsic}, we get the following.

\begin{cor} \label{corzkregseq}
    For $q \geq 1$, we have $\depth_{Z^1} \Ok_Z \geq q$ if and only if 
    $\codim Z^k \geq q + k$ in $Z$ for all $k \geq 1$
\end{cor}

\begin{proof}
    If we apply Theorem~\ref{thmprofsk} to $A = Z^1$ and $\mathcal{F} = \Ok_Z$, then we
    only need to prove that $\codim Z^k \geq q + k$ for $k \geq 1$ is equivalent to $\dim Z^1 \cap S_{k + q}(\Ok_Z) \leq k$.
    We can write the last condition as $\dim (Z^1 \cap Z_{n-k-q}) \leq k$ by Proposition~\ref{propzksk}.
    If we replace $\dim V$ by $n - \codim V$ and set $k' = n-k-q$, we get $\codim (Z^1 \cap Z_{k'}) \geq q + k'$.
    Since $Z_k = Z$ for $k \leq p$, where $p = \codim Z$, and $Z^1 = Z_{p+1}$, this condition for $k \leq p$ is equivalent to
    $\codim Z_{p+1} \geq p + q + 1$ (in $\Omega$), and since $Z_k \subseteq Z_{p+1} = Z^1$ for $k > p + 1$, this is
    equivalent to $\codim Z_{p + k} \geq p + q + k$ for $k \geq 2$.
\end{proof}

In $\Cn$, it is a standard result that a tuple $f = (f_1,\dots,f_p)$ of holomorphic functions is a 
complete intersection if and only if it is a regular sequence (see for example \cite{dJP}, Corollary 4.1.20).
However, Corollary~\ref{corzkregseq} says that this is not always the case on a singular variety.
We will illustrate this with an example.

\begin{ex}
    Let $\pi(t_1,t_2) = (t_1,t_1t_2,t_2^2,t_2^3)$, and let $Z = \pi(\C^2)$. Then $Z_\sing = \{ 0 \}$,
    because outside of $\{ t_1 = t_2 = 0 \}$, one can construct a holomorphic inverse to $\pi$,
    and we will see that $Z$ is not normal at $0$, so $0 \in Z_\sing$.
    The function $f$ such that $\pi^* f = t_2$ is weakly holomorphic on $Z$, since when $t_1 \neq 0$,
    $f = z_2/z_1$, and when $t_2 \neq 0$, $f = z_4/z_3$, so that $f \in \Ok(Z_\reg)$, and it is clear
    that $f$ is locally bounded near $Z_\sing = \{ 0 \}$.
    However, $f$ is not strongly holomorphic at $0$, because if $f = h$ on $Z$ in a neighborhood
    of $0$, where $h$ is holomorphic in a neighborhood of $0$ in $\C^4$, then by taking pull-back
    by $\pi$ to $\C^2$, we get
    \begin{equation*}
        t_2 = h(t_1,t_1t_2,t_2^2,t_2^3),
    \end{equation*}
    which can be seen to be impossible by a Taylor expansion of $h$ at $0$.

    Since $Z$ has pure dimension, $\codim Z^k \geq k + 1$ for $k \geq 1$ by \cite{E},
    Corollary~20.14b. Hence, $Z^k = \emptyset$ for $k \geq 2$.
    Since $Z$ is not normal, it does not satisfy the condition
    \begin{equation} \label{eqnormcond}
        \codim Z^k \geq k + 2, \quad k \geq 0
    \end{equation}
    for normality (see the introduction). However, since $Z^0 = Z_\sing = \{ 0 \}$, the condition \eqref{eqnormcond}
    is satisfied for all $k \neq 1$. Thus, since $Z^1 \subseteq Z_\sing$, and $\codim Z^1 \not\geq 3$,
    we must have $Z^1 = \{ 0 \}$. 
    By Corollary~\ref{corzkregseq}, there does not exist a regular $\Ok_Z$-sequence $f = (f_1,f_2)$ in $\mathcal{J}_{Z^1}$,
    since any such sequence has length $\leq 1$.
    In particular, if we take $f = (z_1,z_3)$, then $f$ is a complete intersection
    since $Z \cap \{ z_1 = z_3 = 0 \} = \{ 0 \}$, but $f$ is not a regular sequence.
    We claim that one can also see this more directly.
    To begin with, it is clear that $z_3 \notin (z_1)$ in $\Ok_Z$ since $Z \cap \{ z_1 = 0 \} \not\subseteq Z \cap \{ z_3 = 0 \}$.
    We also have that $z_2 \notin (z_1)$ in $\Ok_Z$, since if $z_2 \in (z_1)$, then by taking pull-back to
    $\C^2$ as above, we get
    \begin{equation*}
        t_1 t_2 = t_1 h(t_1,t_1t_2,t_2^2,t_2^3),
    \end{equation*}
    which is easily seen to be impossible.
    However, since $z_2 z_3 = z_1 z_4$ in $\Ok_Z$, we get that $ z_2 z_3 \in (z_1)$ in 
    $\Ok_Z$. Thus, $z_3$ is a zero-divisor in $\Ok_Z/(z_1)$, i.e., $(z_1,z_3)$ is not
    a regular $\Ok_Z$-sequence in $\mathcal{J}_{Z^1}$. 
\end{ex}

\begin{lma} \label{lmaannregseq}
    Let $f = (f_1,\dots,f_k)$ be a complete intersection on $(Z,z)$. If
    \begin{equation*}
        \ann\left( \dbar\frac{1}{f_r}\wedge \dots \wedge \dbar \frac{1}{f_1}\right) = \mathcal{J}(f_1,\dots,f_r) \text{ for all $r<k$},
    \end{equation*}
    then $(f_1,\dots,f_k)$ is a regular $\Ok_{Z,z}$-sequence.
\end{lma}

\begin{proof}
    If $k = 1$, this is clear since $\Ok_{Z,z}$ is reduced and $f$ is assumed to be a complete intersection.
    By induction over $k$, we can assume that $(f_1,\dots,f_{k-1})$ is a regular $\Ok_{Z,z}$-sequence.
    Assume that $(f_1,\dots,f_k)$ is not a regular sequence in $\Ok_{Z,z}$. Then,
    since $f_k \notin \mathcal{J}(f_1,\dots,f_{k-1})$, there exist $g \notin \mathcal{J}(f_1,\dots,f_{k-1})$
    such that $f_k g \in \mathcal{J}(f_1,\dots,f_{k-1})$. But since $g \in \mathcal{J}(f_1,\dots,f_{k-1})$ outside
    of $\{ f_k = 0 \}$, we get that
    \begin{equation*}
        \supp \left( g\dbar\frac{1}{f_{k-1}}\wedge \dots \wedge \dbar\frac{1}{f_1} \right) \subseteq \{ f_1 = \dots = f_k = 0 \}
    \end{equation*}
    by Theorem~\ref{thmanninclusion}. But then by Proposition~\ref{proppm0}, we get that
    \begin{equation*}
        g \in \ann \dbar \frac{1}{f_{k-1}}\wedge \dots \wedge \dbar \frac{1}{f_1} = \mathcal{J}(f_1,\ldots,f_{k-1}),
    \end{equation*}
    which is a contradiction.
\end{proof}

\begin{proof}[Proof of Proposition~\ref{propcounterex2}]
    By Lemma~\ref{lmaci2}, there exists a complete intersection $(f_1,\dots,f_{p+1})$ such that
    $Z^1 \subseteq \{ f_1 = \dots = f_{p+1} = 0 \}$.
    By Corollary~\ref{corzkregseq}, $(f_1,\dots,f_{p+1})$ is not a regular $\Ok_{Z,z}$-sequence in $\mathcal{J}(f_1,\dots,f_{p+1})_z$.
    Thus by Lemma~\ref{lmaannregseq}, we must have that
    \begin{equation} \label{eqannstrictinclusion}
        \ann \left( \dbar\frac{1}{f_k}\wedge \dots \wedge \dbar\frac{1}{f_1}\right) \supsetneq \mathcal{J}(f_1,\dots,f_k)
    \end{equation}
    for some $k \leq p$. However, by Theorem~\ref{annmuf}, we have equality for $k \leq {p-1}$.
    Thus we must have strict inclusion in \eqref{eqannstrictinclusion} for $k = p$.
\end{proof}

\section*{Acknowledgments}

I would like to thank my advisor Mats Andersson for valuable discussions
in the preparation of this article.

\end{document}